\documentclass[a4paper, twoside, 11pt]{paper}
\usepackage{amssymb}
\usepackage{amsthm}
\usepackage{amsmath}
\usepackage{amsfonts}
\usepackage[mathscr]{euscript}
\usepackage{enumerate}
\usepackage{graphics}
\usepackage[dvips]{graphicx}

\begin{document}



\let\goth\mathfrak


\newcommand\theoremname{Theorem}
\newcommand\corollaryname{Corollary}
\newcommand\propositionname{Proposition}
\newcommand\factname{Fact}
\newcommand\remarkname{Remark}
\newcommand\notename{Note}
\newcommand\problemname{Problem}

\newtheorem{thm}{\theoremname}[section]
\newtheorem{cor}[thm]{\corollaryname}
\newtheorem{prop}[thm]{\propositionname}
\newtheorem{fact}[thm]{\factname}
\newtheorem{note}[thm]{\notename}
\newtheorem{prob}[thm]{\problemname}
\newtheorem{rema}[thm]{\remarkname}

\newtheorem{cnstrx}[thm]{Construction}
\newenvironment{constr}{\begin{cnstrx}\normalfont}{\myend\end{cnstrx}}
\def\myend{{}\hfill{\small$\bigcirc$}}

\newenvironment{ctext}{%
  \par
  \smallskip
  \centering
}{%
 \par
 \smallskip
 \csname @endpetrue\endcsname
}


\newcounter{sentence}
\def\thesentence{\roman{sentence}}
\def\labelsentence{\upshape(\thesentence)}

\newenvironment{sentences}{%
        \list{\labelsentence}
          {\usecounter{sentence}\def\makelabel##1{\hss\llap{##1}}
            \topsep3pt\leftmargin0pt\itemindent40pt\labelsep8pt}%
  }{%
    \endlist}


\newcommand*{\sub}{\raise.5ex\hbox{\ensuremath{\wp}}}
\newcommand*{\struct}[1]{{\ensuremath{\langle #1 \rangle}}}
\def\id{\mathrm{id}}
\newcommand{\msub}{\mbox{\large$\goth y$}}    
\def\suport{{\mathrm{supp}}}
\def\inc{\mathrel{\strut\rule{3pt}{0pt}\rule{1pt}{9pt}\rule{3pt}{0pt}\strut}}
\def\lines{{\cal L}}

\def\dual#1{{#1}^\circ}

\def\VerSpace(#1,#2){{\bf V}_{{#2}}({#1})}
\def\GrasSpace(#1,#2){{\bf G}_{{#2}}({#1})}
\def\VerSpacex(#1,#2){{\bf V}^\ast_{{#1}}({#2})}

\def\konftyp(#1,#2,#3,#4){\left( {#1}_{#2}\, {#3}_{#4} \right)}

\newcount\liczbaa
\newcount\liczbab
\def\binkonf(#1,#2){\liczbaa=#2 \liczbab=#2 \advance\liczbab by -2
\def\doa{\ifnum\liczbaa = 0\relax \else
\ifnum\liczbaa < 0 \the\liczbaa \else +\the\liczbaa\fi\fi}
\def\dob{\ifnum\liczbab = 0\relax \else
\ifnum\liczbab < 0 \the\liczbab \else +\the\liczbab\fi\fi}
\konftyp(\binom{#1\doa}{2},#1\dob,\binom{#1\doa}{3},3) }
\let\binconf\binkonf

\newcounter{zdanie}\setcounter{zdanie}{0}

\def\ginom(#1,#2){\binom{{#1}+{#2}}{{#1}}}
\def\ginomx(#1,#2){{\sf B}({#1},{#2})}
\def\ginconf(#1,#2){\konftyp({\binom{{#1}+{#2}-{1}}{#1}},{#1},{\binom{{#1}+{#2}-{1}}{#2}},{#2})}
\def\ginconfx(#1,#2){{\mathscr B}({#1},{#2})}
\let\ginkonf\ginconf
\let\ginkonfx\ginconfx

\def\vergras{{\goth R}}
\def\starof(#1){{\mathop{\mathrm S}(#1)}}
\def\topof(#1){{\mathop{\mathrm T}(#1)}}


\title[Pascal Triangles of Configurations]{%
Hyperplanes in Configurations, decompositions, and Pascal Triangle of Configurations}
\author{Krzysztof Pra{\.z}mowski}

%

\maketitle
\begin{abstract}
  An elegant procedure which characterizes a decomposition of some class of binomial configurations
  into two other, resembling a definition of Pascal's Triangle, was given in \cite{gevay}.
  In essence, this construction was already presented in \cite{perspect}. We show that
  such a procedure is a result of fixing in configurations in some class $\mathcal K$ suitable hyperplanes
  which both: are in this class, and deleting such a hyperplane results in a configuration
  in this class. By a way of example we show two more (added to that of \cite{gevay})
  natural classes of such configurations, discuss some other, and propose some open questions
  that seem also natural in this context.
\par\noindent
  Mathematics Subject Classification: 05B30, 51E30 (51E20)
\par\noindent
  Keywords: Pascal Triangle (of binomials), binomial, configuration, hyperplane, combinatorial
  Grassmannian, combinatorial Veronesian, Pascal Triangle of Configurations
\end{abstract}

\section*{Introduction}\label{sec:intro}

On one hand, ``Pascal Triangle'' is a term which is known to all mathematicians: it characterizes
an arrangement of binomial coefficients in a form of a `pyramid' such that each item is the sum
of items placed immediately above it. In another view: the sum of each neighbour items in a row equals to
the item which is their common neighbour (in the  row below).
Clearly, binomial coefficients are simply values of a two-argument function $b(n,k)$ defined on nonnegative integers
($n=0,1,\ldots$, $k=0,\ldots,n$)
and nothing `magic' is in the pyramid defined above. It is a visual presentation of recursive
equation which these coefficients satisfy. Clearly, the sequences of boundary values $b(n,0)$ and $b(n,n)$
uniquely determine then the function $b$. Nevertheless the recurrence in question is extremely simple...

Quite recently, Gabor G{\'e}vay in \cite{gevay} noted that there is family of point-line configurations
which can be arranged in such a pyramid, with a suitably defined ``sum'' of the configurations in question.
Or: each (nontrivial, non-boundary) configuration in this family can be decomposed into two other members
of this family. In essence, this decomposition (even in a more general form) was presented also earlier 
in \cite[Representation 2.12]{perspect};
the class in question consists of configurations which generalize Desargues configuration considered as schemes of
mutual perspectives between several simplexes. On other hand, such systems of (geometrical) perspectives can be found 
even in the classical book of Veblen and Young \cite{class:proj} (G{\'e}vay quotes also explicitly Danzer and Cayley)
and its combinatorial schemes are special instances
of so called binomial graphs, investigated in the context of association schemes (cf. e.g. \cite{klin}),
and associated incidence structures. Combinatorial schemes characterizing these configurations can be found
already in \cite{levi} and \cite{coxet}.
So: 
\begin{quotation}
  the subject was known, but its regular nature was not known -- was not stated explicitly until \cite{gevay}.
\end{quotation}
But then it appeared that the ``sum'' of two configurations is not a well defined operation that depends 
solely on the summands, and the associated decomposition is, in fact, associated with a choice of a hyperplane
in the decomposed configuration. After that become clear (we present these observation in
Section \ref{sec:binconfy}, Theorem \ref{thm:decompo0} and equation \eqref{eq:decomp0}) there appeared
that there are other natural known classes of configurations that can be arranged into respective
triangles. These are, in particular, so called combinatorial Veronesians (defined originally in \cite{combver},
without any connections with studying hyperplanes in configurations). In Section \ref{sec:exm}
we discuss some of the classes which appear within this theory.

\section{Notations, standard constructions}\label{sec:nota}
\subsection{Elementary combinatorics}\label{ssec:intro:combin}

There are well known formulas concerning binomial coefficients, frequently
referred to as ``Pascal Triangle of Binomials". To be more precise, 
these formulas correspond
to the arrangement of the binomial coefficients in a pyramid with consecutive rows:
\begin{ctext}
$\Big( \left(\binom{n}{k}\colon k=0,\ldots,n \right)\colon n = 0,1,2,\ldots \Big)$.
\end{ctext}
Then the corresponding recursive formula is the following
\begin{eqnarray}\label{pyramid:0}
  \binom{n}{k} & = & \binom{n-1}{k-1} + \binom{n-1}{k};
\end{eqnarray}
equation \eqref{pyramid:0} yields immediately next two:
\begin{eqnarray}\label{pyramid:1}
  \textstyle{\binom{n}{k} - \binom{n-1}{k}} & = & \textstyle{\binom{n-1}{k-1}}, \text{ and}
  \\ \label{pyramid:2}
  \textstyle{\binom{n}{k} - \binom{n-1}{k-1}} & = & \textstyle{\binom{n-1}{k}}.
\end{eqnarray}

For purposes of our next investigations it will be more convenient to arrange 
binomial coefficients into a (infinite) matrix:
\begin{ctext}
  $\big[ \ginomx(m,k)\colon m,k = 0,1,\ldots \big]$,
\end{ctext}
where
\begin{equation}
  \ginomx(m,k) = \ginom(m,k);
\end{equation}
clearly, $\ginomx(m,k) = \ginomx(k,m)$; the fundamental recursive formula for 
the binomial coefficients takes the form
\begin{equation}\label{rec:ginom}
  \ginomx(m,k) = \ginomx(m,k-1) + \ginomx(m-1,k).
\end{equation}
%

\subsection{Rudiments of geometry of configurations}\label{ssec:intro:config}

We say that a structure 
${\goth K} = \struct{U,\lines,\inc}$ with $\inc\; \subset U\times\lines$ 
is a {\em $\konftyp(\nu,\rho,\beta,\kappa)$-configuration} 
if $\goth K$ is a partial linear space
(i.e. $a,b \inc A,B$ yields $a =b$ or $A = B$) 
such that $|U| = \nu$, $|\lines| = \beta$, exactly $\rho$ elements of $\lines$
are in the relation $\inc$ with $a\in U$, for each $a\in U$, and exactly
$\kappa$ elements of $U$ are in the relation $\inc$ with $A \in \lines$, for each
$A\in\lines$.

Let $\goth K$ be a configuration as above, then the following equation (a specialized form
of the so called fundament equation of partial linear spaces) holds
\begin{equation}\label{equ:pls}
  \nu \cdot \rho = \beta \cdot \kappa.
\end{equation}
The elements of $U$ are called {\em points} of $\goth K$, the elements of $\lines$
are called {\em lines} of $\goth K$, and the relation $\inc$ is {\em the incidence}.
The numbers $\rho$ and $\kappa$ are referred to as {\em point rank} and
{\em line size/rank} resp.
It is a folklore, that every configuration as above with $\kappa\geq 2$ is isomorphic
to a configuration, whose lines are sets of points, and the incidence is 
the standard membership relation $\in$.
If this will not cause a confusion (as it may happen in particular examples) 
we shall frequently assume that 
the incidence of $\goth K$ is the membership relation.

A subset $\cal H$ of the set of points of $\goth K$ is called {\em a hyperplane} of $\goth K$
when 
\begin{itemize}\def\labelitemi{--}\itemsep-2pt
\item $\cal H$ is {\em a subspace} of $\goth K$, i.e. if 
  the conditions $a,b \inc A\in\lines$ and $a,b\in{\cal H}$, $a\neq b$
  yield $x \in {\cal H}$ for every $x$ such that $x \inc A$, 
\item[\strut] and
\item
  each line of $\goth K$ crosses $\cal H$, i.e. for each $A\in\lines$ there is 
  $x\in {\cal H}$ such that $x\inc A$.
\end{itemize}
Let $\cal H$ be a hyperplane of $\goth K$.
Then, for each line $A$ of $\goth K$ either there is a unique $x\in{\cal H}$ with
$x \inc A$ (we write $x = A^\infty$ in that case) or every point incident with $A$
belongs to $\cal H$: the set of such lines will be denoted by $\lines[{\cal H}]$.
Clearly, 
\begin{ctext}
  ${\goth K}\restriction{\cal H} := \struct{{\cal H},\lines[{\cal H}],%
   \inc \cap \big({\cal H}\times\lines[{\cal H}]\big) }$
\end{ctext}
is a partial linear space; quite frequently in the sequel we shall make no distinction between 
$\cal H$ and ${\goth K}\restriction{\cal H}$.
Clearly, the set $U$ of all the points of $\goth K$ is a hyperplane of $\goth K$.
In what follows we shall assume that a hyperplane means a {\em proper} (i.e. ${\cal H}\neq U$)
subspace that satisfies suitable conditions.

Given a hyperplane $\cal H$ of $\goth H$ we define {\em the reduct}
\begin{ctext}
  ${\goth K} \setminus {\cal H} :=
    \struct{ U\setminus {\cal H},\, \lines\setminus \lines[{\cal H}],\,%
    \inc \cap \big((U\setminus {\cal H})\times(\lines\setminus \lines[{\cal H}])\big)}$;
\end{ctext}
if $\kappa \geq 3$ then ${\goth K}\setminus {\cal H}$ is a partial linear
space with all the lines of size (rank) $\kappa -1$. 
Let us write, for symmetry, ${\goth K}_1 = {\goth K}\setminus {\cal H}$
and ${\goth K}_2 = {\goth K}\restriction{\cal H}$.
Recall, that we have a function $\infty$ from the lines 
of ${\goth K}_1$ into the points of ${\goth K}_2$.
Let us try to ``reverse'' this decomposition:
\begin{constr}\label{def:zlepka}
  Let ${\goth K}_i = \struct{U_i,\lines_i,\inc_i}$ be a partial linear space for $i=1,2$.
  Assume that $U_1 \cap U_2 = \emptyset = \lines_1 \cap \lines_2$.
  Let $\infty\colon\lines_1\longrightarrow U_2$ be a map such that the following
  holds
  \begin{ctext}
    if $U_1\ni x \inc A,B\in\lines_1$ and $A^\infty = B^\infty$ then $A = B$.
  \end{ctext}
  We define 
  \begin{description}\itemsep-2pt
  \item[$U:$]      $= U_1 \cup U_2$, 
  \item[$\lines:$] $= \lines_1\cup\lines_2$,
  \item[$\inc:$]   $= \inc_1 \cup \inc_2 \cup \{(x,A) \colon U_2\ni x=A^\infty, A\in\lines_1\}$.
  \end{description}
  Finally, we set 
  \begin{equation}\label{def:zlepka0}
    {\goth K}_1 \rtimes_\infty {\goth K}_2 := \struct{U,\lines,\inc}.
  \end{equation}
  It is evident that 
  {\em ${\goth K}_1 \rtimes_\infty {\goth K}_2$ is a partial linear space}.
\end{constr}
\begin{prop}
  Let ${\goth K} = {\goth K}_1 \rtimes_\infty {\goth K}_2$ with ${\goth K}_i$
  as in \ref{def:zlepka}.
  Then $U_2$ is a hyperplane in $\goth K$ and $\lines_2 = \lines[U_2]$.
\end{prop}
\begin{proof}
  It suffices to state directly that if $A\in\lines$ then either 
  $A\in\lines_2$ and then $x\inc A$ gives $x\in U_2$, or $A\in\lines_1$ and then $x\inc A$
  yields $x\in U_1$ or $U_2 \ni x = \infty(A)$.
\end{proof}
The construction of the type \ref{def:zlepka} is quite frequent in geometry.
One particular case let us mention below:
\begin{note}\normalfont 
  Let ${\goth K}_1 = \struct{U_1,\lines_1,\inc_1,\parallel_1}$ be a partial linear
  space with parallelism of lines; we write $[A]_{\parallel_1}$ for the equivalence
  class of $A\in\lines_1$ w.r.t. the relation $\parallel_1$ 
  (i.e. simply for the direction of $A$).
  Suppose that there is a formula $\Phi$ in the 
  language of ${\goth K}_1$ such that the relation 
  \begin{ctext}
    $\{ ([A_1]_{\parallel_1},[A_2]_{\parallel_1},[A_3]_{\parallel_1})\colon %
    (A_1,A_2,A_3)\in\lines_1^3,\ \Phi(A_1,A_2,A_3) \}$
  \end{ctext}
  is a ternary equivalence relation on the set $\big( \lines_1\diagup\parallel_1 \big)^3$
  (cf. \cite{ternequiv});
  let $\lines_2$ be the set of its equivalence classes, and 
  ${\goth K}_2 = \struct{{\lines_1\diagup\parallel_1},\lines_2,\in}$.
  With $A^\infty = [A]_{\parallel_1}$ for $A\in\lines_1$ we obtain the structure
  ${\goth K} = {\goth K}_1 \rtimes_\infty {\goth K}_2$ which is called, in that context, the 
  {\em closure of an affine structure} ${\goth K}_1$.
  \par
  In particular cases of this construction, practically, the structures 
  ${\goth K}_1$ and $\goth K$ are given, and 
  {\em we search for an appropriate formula $\Phi$}
  (see \cite{afclos}: affine completion, \cite{polarclos}, \cite{segreclos}).
\myend
\end{note}
Other examples of this construction will appear in the next Section.

\subsection{Dualization}\label{ssec:duale}

Let ${\goth K} = \struct{U,\lines,\inc}$ be an incidence structure;
we call the structure
\begin{ctext}
  $\dual{{\goth K}} = \struct{\lines,U,\inc^{-1}}$
\end{ctext}
the dual of $\goth K$. It is evident that $\dual{{\goth K}}$ is a partial linear space
whenever $\goth K$ is so. In particular
\begin{ctext}
  if $\goth K$ is a $\konftyp(\nu,\rho,\beta,\kappa)$-configuration then
  $\dual{\goth K}$ is a $\konftyp(\beta,\kappa,\nu,\rho)$-configuration.
\end{ctext}
\begin{prop}\label{prop:dual-hypy}
  Let $\cal H$ be a hyperplane of a partial linear space ${\goth K} = \struct{U,\lines}$
  such that the induced correspondence $\infty$ is bijective.
  Then $\lines\setminus\lines[{\cal H}]$ is a hyperplane of $\dual{\goth K}$.
\end{prop}
\begin{proof}
  Let $L_1,L_2 \in \lines\setminus\lines[{\cal H}]$. Assume that 
  $L_1,L_2\inc^{-1} x\in U$  and $\lines\ni L \inc^{-1} x$.
  Suppose that $L\notin\lines\setminus\lines[{\cal H}]$, then $L\in\lines[{\cal H}]$
  and, consequently, $x\in{\cal H}$. This gives $x = \infty(L_1) = \infty(L_2)$;
  we have $L_1 = L_2$ then. This proves that $\lines[{\cal H}]$
  is a subspace of $\dual{\goth K}$.
  \par
  Let $L$ be an arbitrary line of $\dual{\goth K}$, then $L \in U$. If $L\notin{\cal H}$
  then each line of $\goth K$ (each point of $\dual{\goth K}$) that passes through $L$
  is in $\lines\setminus\lines[{\cal H}]$. If $L\in{\cal H}$ then 
  $\infty^{-1}(L) \inc^{-1} L$. This suffices for the proof.
\end{proof}
Standard examples show that the condition {\em $\infty$ is bijective} 
assumed in \ref{prop:dual-hypy} cannot be removed. Indeed, the plane in a 
projective 3-space $\goth P$
is a hyperplane, but the family of lines of the resulting affine 3-space is not
even a subspace of $\dual{\goth P}$.
However, \ref{prop:dual-hypy} appears useful when we deal with (binomial) configurations.
Proposition \ref{prop:dual-hypy} can be easily (re)formulated in a more `constructive'
fashion:
\begin{cor}\label{cor:dual-hipy}
  Let ${\goth K}_i$ be configurations as in \ref{def:zlepka} with a suitable map 
  $\infty$ defined. Assume that $\infty$ is a bijection and
  ${\goth K} = {\goth K}_1\rtimes_\infty {\goth K}_2$.
  Then
  \begin{equation}\label{wzor:dual-hipy}
    \dual{\goth K}\quad = \quad \dual{{\goth K}_2} \rtimes_{\infty^{-1}} \dual{{\goth K}_1}
  \end{equation}
\end{cor}

\section{Binomial configurations}\label{sec:binconfy}

\subsection{Generalities}

The main subject of this section consists in investigations
on the family of {\em binomial configurations}
i.e. of configurations of the type 
$\ginconf(k,m)$ for some positive integers $k,m$.
It is easily seen that each parameters of this form satisfy \eqref{equ:pls}.
Let us write
\begin{ctext}
  $\ginkonfx(k,m)$ for the class of all $\ginconf(k,m)$-configurations.
\end{ctext}

\begin{thm}\label{thm:decompo0}
  Let ${\goth K}\in\ginkonfx(k,m)$ and let $\cal H$ be a hyperplane of $\goth K$.
  Assume that 
  \begin{sentences}\itemsep-2pt
  \item\label{war1}
    $\cal H$ is a configuration (in this case this means simply that 
    ${\goth K}\restriction{\cal H}$ has constant point rank),
    and
  \item\label{war2}
    ${\goth K}\setminus{\cal H}$ is a binomial configuration.
  \setcounter{zdanie}{\value{sentence}}
  \end{sentences}
  Then
  \begin{sentences}\itemsep-2pt\setcounter{sentence}{\value{zdanie}}
  \item\label{war3}
    $\cal H$ is a binomial configuration, more precisely:
      ${\goth K}_2 = {\goth K}\restriction{\cal H}\in\ginkonfx(k-1,m)$;
  \item\label{war4}
    ${\goth K}_1 = {\goth K}\setminus{\cal H}\in\ginkonfx(k,m-1)$;
  \item\label{war5}
    there is a 1-1 correspondence 
    $\infty\colon\text{lines of }{\goth K}_1 \longrightarrow\text{ points of }{\goth K}_2$
    such that 
    ${\goth K} = {\goth K}_1 \rtimes_\infty {\goth K}_2$.
  \setcounter{zdanie}{\value{sentence}}
  \end{sentences}
\end{thm}
\begin{proof}
  Recall that, right from the definition, the points of $\goth K$ have rank $k$, 
  and the lines of $\goth K$ have size $m$.
  Set $n=m+k-1$.
  
  Then, 
  from the definition we get immediately that 
  the points of ${\goth K}_1$ are all of the same rank $k$ and 
  the lines are all of the size $m-1$, 
  so, in accordance with \eqref{war2},
  ${\goth K}_1$ is a $\ginconf(k,m-1)$-configuration, which justifies \eqref{war4}.
  The number of points in ${\goth K}_1$ is $\binom{n-1}{k}$ and 
  the number of points of $\goth K$ is $\binom{n}{k}$; 
  from the Pascalian equations the number of points 
  of ${\goth K}_2$ is $\binom{n}{k} - \binom{n-1}{k} = \binom{n-1}{k-1}$.
  Similarly we compute the number of lines of ${\goth K}_2$:
  it equals to $\binom{n}{m} - \binom{n-1}{m-1} = \binom{n-1}{m}$.
  The size of the lines in $\cal H$ is $m$; 
  from assumption \eqref{war1} and \eqref{equ:pls}
  applied to ${\goth K}_2$ we get that the point rank in ${\goth K}_2$ equals to $k-1$.
  So, ${\goth K}_2$ is a $\ginconf(k-1,m)$-configuration.
  This justifies \eqref{war3}.
  Finally, since each point in ${\goth K}_2$ has its rank on one less than in $\goth K$
  we get that through each one of these points there passes exactly one line of ${\goth K}_1$,
  so $\infty$ is a bijection, as required in \eqref{war5}.
\end{proof}
Informally speaking, \ref{thm:decompo0} gives a decomposition
\begin{equation}\label{eq:decomp0}
  \ginconfx(k,m) = \ginconfx(k,m-1) \rtimes_\infty \ginconfx(k-1,m),
\end{equation}
which resembles reverent Pascalian equation \eqref{rec:ginom}.
But note, that the ``operation'' $\rtimes_\infty$ is not commutative, and 
it depends essentially on the parameter $\infty$.

\begin{rema}\normalfont
  {\it Not every hyperplane of a binomial configuration is a (binomial) configuration.}
  Indeed, it suffices to have a look on hyperplanes in binomial partial Steiner
  triple systems, either in a more general approach of \cite{hypinbin:psts} or
  in a more particular case of \cite{hypingendes} and note that in the Desargues
  configuration 
  a line accomplished with a point not joinable with any point on this line 
  is a hyperplane, it contains three points of rank $3$ and one point of rank $0$
  so, it is not a configuration.
\myend
\end{rema}
\begin{rema}\normalfont
  Let us consider the smallest sensible and possible case:
  $\ginconfx(2,3) \rtimes \ginconfx(3,2) = \ginconfx(3,3)$.
  If ${\goth K}\in\ginconfx(3,3)$ then $\goth K$ is a 
  $\ginconf(3,3)= \konftyp(10,3,10,3)$-configuration: one of
  ten possible.
  If ${\goth K}_1$ is a $\ginconf(3,2)=\konftyp(4,3,6,2)$-configuration then
  it is the complete graph $K_4$.
  If ${\goth K}_2$ is a $\ginconf(2,3)=\konftyp(6,2,4,3)$-configuration then
  it is simply the Pasch-Veblen configuration $\goth V$.
  It was shown in \cite{klik:VC} that there are exactly six maps $\infty$ which yield
  pair wise non isomorphic configurations $K_4 \rtimes_\infty {\goth V}$.
  So, {\em there are binomial configurations ${\goth K}_1,{\goth K}_2$
  and bijections 
    $\infty',\infty''\colon\text{lines of }{\goth K}_1 \longrightarrow\text{ points of }{\goth K}_2$
  such that 
  ${\goth K}_1\rtimes_{\infty'}{\goth K}_2 \not\cong {\goth K}_1\rtimes_{\infty''}{\goth K}_2$.}
  Consequently, the symbol $\rtimes$ {\em is not a well defined operation,
  without the argument $\infty$ defined explicitly}.
\end{rema}
\begin{rema}\normalfont
  Let ${\goth K}_1,{\goth K}_2$ be binomial configurations, let a  map
    $\infty\colon\text{lines of }{\goth K}_1 \longrightarrow\text{ points of }{\goth K}_2$
  be a bijection. 
%
\par
  From assumption, ${\goth K}_i\in \ginconfx(k_i,m_i)$ for some integers $k_i,m_i$, $i=1,2$.
  Moreover, the two numbers: of lines of ${\goth K}_1$ and of points of ${\goth K}_2$
  coincide. This means than
  $\binom{k_1+m_1-1}{m_1} = \binom{k_2+m_2-1}{k_2}$.
  Then $k_1+m_1-1 = k_2+m_2-1$ and one of the following holds:
  \begin{sentences}
  \item
    either $m_1 = m_2-1$ -- in this case $k_2 = k_1-1$ and
    {\it ${\goth K} = {\goth K}_1 \rtimes_\infty{\goth K}_2$ is a binomial configuration},
  \item
    or $m_1 = k_2$ and then $k_1 = m_2$. Consider e.g. the case $k_1=m_1=k_2=m_2=3$, then
    ${\goth K}_i$ are $\konftyp(10,3,10,3)$-configurations. But then
    ${\goth K} = {\goth K}_1 \rtimes_\infty{\goth K}_2$ has $20$ points and $20$
    lines. Ten lines have size $3$, and ten have size $4$. So, in this case 
    {\it $\goth K$ is not even a configuration.}
  \end{sentences}
  This shows that a `sum' of two binomial configurations, even determined by constructing `improper
  points', may be not a binomial configuration.\myend
\end{rema}

In the next Section we present two remarkable families of binomial configurations
which yield  families indexed by positive integers and which yield ``a Pascal Triangle".

\section{Examples}\label{sec:exm}
\subsection{Example: the family of combinatorial Grassmannians}\label{exm:grasy}

For an integer $k$ and a set $X$ we write $\sub_k(X)$ for the family of $k$-subsets
of $X$.
Nowadays the  notation $\binom{X}{k}$ instead of $\sub_k(X)$ becomes widely used.
We prefer, however, not to mix integers and sets.

Let ${\goth K}\in\ginconfx(k,m)$; then the points of $\goth K$ can be identified with 
the $k$-subsets of a fixed $n$-element set $X$, where $n = m+k-1$.
Let us identify the lines of $\goth K$ with the elements of $\sub_m(X)$ and
define 
\begin{equation}\label{def:inc:gras0}
  a \inc A :\iff a \in \sub_k(X) \land A \in \sub_m(X) \land |a \cap A|=1.
\end{equation}
Suppose that $a\neq b$ and $a,b \inc A$ with $\inc$ defined by \eqref{def:inc:gras0}.
Then $a\cap b = X\setminus A$ and therefore $A$ is uniquely determined by its 
two points $a$ and $b$. So, the structure
\begin{ctext}
  ${\goth G}(k,m) := \struct{\sub_k(X),\sub_m(X),\inc}$
\end{ctext}
is a partial linear space. It is not too hard to verify that it is
a configuration with the lines of size $m$ and the points of rank $k$,
so ${\goth G}(k,m)\in\ginconfx(k,m)$.

In practice, the above presentation is not so easy to handle with and not too intuitive.
\begin{sentences} 
\item
  There is a one-to-one correspondence between the elements of $\sub_m(X)$ and 
  the elements of $\sub_{k-1}(X)$: indeed, $n = m + (k-1)$ so, the boolean complementation
  $\varkappa$ is a bijection in question. Then we see that the pair of maps
  $(\id,\varkappa)$ maps ${\goth G}(k,m)$ onto the structure
  $\struct{\sub_k(X),\sub_{k-1}(X),\supset}$, which coincides with the $DCD(n,k)$
  introduced in \cite{gevay}.
\item
  Analogously, there is a one-to-one correspondence between the elements of 
  $\sub_{m-1}(X)$ and the elements of $\sub_k(X)$; set $k_0 = m-1$, then
  $(\varkappa,\id)$ maps ${\goth G}(m,k)$ onto the structure
  $\struct{\sub_{k_0}(X),\sub_{k_0+1}(X),\subset}$, which coincides with the 
  {\em combinatorial Grassmannian} $\GrasSpace(X,k_0)$ defined in \cite{perspect}.
\end{sentences}

Let us concentrate upon the presentation given in \cite{perspect}, let us 
drop out the superfluous index $0$ and let 
${\goth K} = \GrasSpace(X,k)$, $|X| = n$; remember that 
$\GrasSpace(X,k) \in \ginconfx(n-k,k+1)$.
We write $\GrasSpace(n,k)$ for the type of $\GrasSpace(X,k)$ where $|X| = n$.

Let us fix an element $i\in X$, then $\sub_k(X)$ is the disjoint union
$\sub_k(X) = {\cal X}_1 \cup {\cal X}_2$, where 
${\cal X}_1 = \{ a\in \sub_k(X)\colon i\in a \}$ and
${\cal X}_2 = \{ a\in\sub_k(x)\colon i \notin a\} = \sub_k(X\setminus\{i\})$.
The following is easily seen:
\begin{sentences}
\item
  ${\cal X}_2$ is a hyperplane of $\goth K$, 
  ${\goth K}_2 := {\goth K}\restriction{{\cal X}_2} = \GrasSpace(X\setminus\{i\},k)$
\item
  ${\goth K}_1 = {\goth K}\setminus{\cal X}_2$, with the point-set ${\cal X}_1$,
  is isomorphic under the map 
  ${\cal X}_1\ni a \longmapsto a\setminus\{ i \}\in\sub_{k-1}(X\setminus\{ i \})$
  to the structure $\GrasSpace(X\setminus\{i\},k-1)$.
\item\label{jawne1}
  Let $A$ be a line of ${\goth K}_1$, so $A\in\sub_{k+1}(X)$ where $i\in A$.
  Then $A \setminus \{ i \}\in \sub_k(A)\cap {\cal X}_2$, so $A^\infty = A\setminus\{i\}$.
\end{sentences}
In view of the above and \ref{thm:decompo0} we get that
\begin{prop}
  If $i \in X$ is arbitrary then
  \begin{equation}
    \GrasSpace(X,k) \cong \GrasSpace(X\setminus\{i\},k-1)\rtimes_\infty\GrasSpace(X\setminus\{i\},k)
  \end{equation}  
  with $\infty$ defined by \eqref{jawne1} above.
\end{prop}
In numerical symbols we can write:
$$\GrasSpace(n,k) = \GrasSpace(n-1,k-1)\rtimes_\infty\GrasSpace(n-1,k).$$
This decomposition was studied in many details in \cite{gevay}, it was also
noticed in \cite[Representation 2.12]{perspect}.
While expressed in terms of ${\goth G}(k,m)$ it assumes the form
$${\goth G}(k_0,m_0) = {\goth G}(k_0,m_0-1)\rtimes_\infty{\goth G}(k_0-1,m_0),$$
where $k_0 = n-k$, $m_0 = k+1$.

\subsection{Example: the family of combinatorial Veronesians}\label{exm:very}

Let $X$ be an $m$-element set; we write $\msub_k(X)$ for the $k$-element 
multisets with the elements in $X$. In  naive words, a multiset is a `set'
whose elements belong to $X$, and each one of them can occur several times.
Formally, it is a function $f$ defined on $X$ with values in the set of natural numbers
(with zero); this function `counts' how many times given item from $X$ occurs in $f$.
It is a convenient way to symbolize such a function $f$ in the form 
$f = \prod_{x\in X} x^{f(x)}$ (with the natural relations like $x^ix^j = x^{i+j}$,
$x^i y^j = y^j  x^i$, $x^0 = 1$, $1 x = x$, etc...).
Then the cardinality of $f$ is $|f| = \sum_{x\in X}f(x)$.
We write $\suport(f) = \{ x\in X\colon f(x) > 0 \}$;
clearly, $|f| = \sum_{x\in\suport(f)} f(x)$

Let us write  $\bigcup_{i=0}^{i=k-1} \msub_i(X) =: \msub_{<k}(X)$.
On the set $\msub_k(X) \times \msub_{<k}(X)$ we define the incidence relation $\inc$
by the formula:
\begin{equation}\label{def:inc:ver0}
  e \inc f :\iff f = e\, x^{k-|e|} \text{ for some } x\in X.
\end{equation}
The structure 
\begin{ctext}
  $\VerSpace(m,k) = \struct{\msub_k(X),\msub_{<k}(X),\inc}$
\end{ctext}
is called a {\em combinatorial Veronesian}; the class of combinatorial Veronesians
was introduced in \cite{combver}. 
It was proved that $\VerSpace(m,k)$ is a partial linear space with the points of rank
$k$ and the lines of size $m$; the formulas counting the cardinality of $\msub_k(X)$ and
of $\msub_{<k}(X)$ are known in the elementary combinatorics; summing up we get that
$\VerSpace(m,k)\in\ginconfx(k,m)$.

Let us fix $a\in X$ and define
${\cal X}_2 = \{f\in\msub_k(X)\colon a\in\suport(f)\}$
and 
${\cal X}_1 = \{f\in\msub_k(X)\colon a\notin\suport(f)\}$;
then $\msub_k(X)$ is the disjoint union ${\cal X}_1 \cup {\cal X}_2$.

\begin{sentences}
\item
  It is seen that the map $\msub_{k-1}(X)\ni f\longmapsto f\,a^1\in {\cal X}_2$
  is a bijection. Suppose that $f' a^1, f'' a^1 \inc e$ where $e\in\msub_{<k}(X)$.
  Then $a\in\suport(e)$ and $f',f''\inc \frac{e}{a}\in\msub_{<k-1}(X)$.
  Finally, $a\in\suport(f)$ for every $f$ with $f\inc e$, which yields that
  ${\cal X}_2$ is a subspace of $\VerSpace(X,k)$; 
  as we noted, it is isomorphic to $\VerSpace(X,k-1)$.
\item
  Let $e\in\msub_{<k}(X)$ be a line of $\VerSpace(m,k)$. If $a \in\suport(e)$ then 
  $f\in{\cal X}_2$ for every $f$ with $f\inc e$.
  If $a\notin\suport(e)$ then $e^\infty = e\,a^{k-|e|}$ is the unique element incident with
  $e$ which belongs to ${\cal X}_2$.
\item\label{jawne2}
  Evidently, the points in ${\cal X}_1$ can be considered as the points
  of $\VerSpace(X\setminus\{a\},k)$.
  Let $e\in\msub_{<k}(X\setminus\{ a \})$ be a line of $\VerSpace(X\setminus\{a\},k)$;
  then $e^\infty = e\,a^{k-|e|}\inc e$ is well defined. 
\item
  In particular, the above yields that ${\cal X}_2$ is a hyperplane of $\VerSpace(m,k)$.
\end{sentences}
Summing up, we obtain
\begin{prop}
  Let $a\in X$ be arbitrary.
  \begin{equation} 
    \VerSpace(X,k) = \VerSpace(X\setminus\{ a\},k) \rtimes_\infty \VerSpace(X,k-1),
  \end{equation}  
  where $\infty$ is defined by \eqref{jawne2} above.
\end{prop}
In (numerical) symbols we can express this fact by
$$
\VerSpace(m,k) = \VerSpace(m-1,k)\rtimes_\infty \VerSpace(m,k-1).
$$


As a consequence of \cite[Cor. 4.8, Thm. 4.5]{combver}, $\VerSpace(m,k)$
is a combinatorial Grassmannian only for $k=2$ or $m=2$ so,
Grassmannians and Veronesians are essentially distinct families.

\subsection{Example: the family of dual combinatorial Veronesians}\label{exm:duvery}

In Subsections \ref{exm:grasy} and \ref{exm:very},
we have found decompositions of the scheme
  $\ginconfx(k,m) = \ginconfx(k,m-1)\rtimes\ginconfx(k-1,m)$. Clearly, 
  $\ginconfx(m,k)$ are dual to $\ginconfx(k,m)$; 
therefore, in view of \ref{cor:dual-hipy} one can expect that each of these
decompositions determines a decomposition of the scheme
\begin{ctext}
  $\ginconfx(m,k) =\dual{\ginconfx(k,m)} = 
  \dual{\ginconfx(k-1,m)} \rtimes \dual{\ginconfx(k,m-1)} =
  \ginconfx(m,k-1)\rtimes \ginconfx(m-1,k)$
\end{ctext}
In case of combinatorial Grassmannians the dualization procedure does not yield any 
new family of configurations:
\begin{fact}
  Let $n = |X|$ for a set $X$. Then 
  $\dual{\GrasSpace(X,k)} \cong \GrasSpace(X,n-k)$.
\end{fact}
However, the dual Veronesians yield another, third family: if $\dual{\VerSpace(m,k)}$
is (isomorphic to) a combinatorial Grassmannian then either $k=2$ or $m=2$;
if it is isomorphic to a combinatorial Veronesian then $k=2$, or $m=2$, or $k=3=m$.
Even $\dual{\VerSpace(k,k)} \cong \VerSpace(k,k)$ is not valid for $k > 3$
(see \cite[Thm.'s 4.14, 4.15]{combver})!

Let us adopt notation of Subsection \ref{exm:very} and let
${\goth K} = \struct{U,\lines} = \VerSpace(X,k)$; let us remind that 
${\cal X}_2 = \{ f\in\msub_k(X)\colon a \in \suport(f) \}$ is a hyperplane of $\goth K$
and then 
$\lines[{{\cal X}_2}] = \{ e\in \msub_{<k}(X)\colon a \in\suport(e) \} =:\lines_2$.
Consequently, $\lines_1 := \lines\setminus\lines_2 = \msub_{<k}(X\setminus\{ a\})$ is
a hyperplane of $\dual{\goth K}$; set
${\cal X}_1 := U\setminus{\cal X}_2 = \msub_k(X\setminus\{ a \})$.
Consider a line $f\in\msub_{k}(X)$ of $\struct{\lines_2,{\cal X}_2}$;
then $a \in \suport(f)$: let $dg(a,f)$ be the greatest integer $s$ such that $f=a^s g$
for a multiset $g$.
We associate with such an $f$ the point $f^\infty = \frac{f}{a^{dg(a,f)}}\in\lines_1$,
it is seen that we obtain
\begin{prop}
$\dual{\VerSpace(m,k)} = 
\struct{\lines_2,{\cal X}_2,\inc^{-1}} \rtimes_{\infty}\struct{\lines_1,{\cal X}_1,\inc^{-1}}
\cong \dual{\VerSpace(m,k-1)} \rtimes_{\infty} \dual{\VerSpace(m-1,k)}$.
\end{prop}

With the symbols $\VerSpacex(m,k) = \dual{\VerSpace(m,k)}\in\ginconfx(m,k)$ 
we arrive to
$$
\VerSpacex(m,k) = \VerSpacex(m,k-1) \rtimes_\infty \VerSpacex(m-1,k)
$$
Consequently, following \ref{cor:dual-hipy} we can explicitly characterize the Pascal 
Triangle of Configurations consisting of dual combinatorial Veronesians.
%


\section{Comments and problems}

We have shown three families $\mathscr K$ of configurations 
${\goth K}(m,k)\colon m,k=1, 2\ldots$ such that the formula 
${\goth K}(m,k) = {\goth K}(m,k-1)\rtimes_{\infty_{m,k}}{\goth K}(m-1,k)$
is valid for all $m,k$ and suitable maps $\infty_{m,k}$.
One can expect that there are more such families: the point is to find a
suitable family 
\begin{ctext}
  $\big[\infty_{m,k}\colon\sub_{m-1}(m+k-2)\longrightarrow\sub_{k-1}(m+k-2)\colon %
  m,k=1,2,\ldots \big]$
\end{ctext}
It is seen how huge variety of binomial partial triple systems can be obtained
via `completing' complete graphs (see \cite{hypinbin:psts}): 
one can expect that our procedure produces
much more required configurations (cf. Problem \ref{prob:rozmnoz}).

However, one essential question appears: which of them can be realized
in a Desarguesian projective space: we call them {\em projective} then. 
It is known that all the combinatorial Grassmannians
are projective. It is also known that (practically all) combinatorial Veronesians are not
projective (only $\VerSpace(3,3)$ and $\VerSpace(2,k)$, $\VerSpace(m,2)$ are realizable).
Similarly, dual of combinatorial Veronesians are also not projective (besides
the exceptions indicated before), \cite[Thm.'s 6.9, 6.10]{combver}.

The statement like 
{\em if ${\goth K}_1$ and ${\goth K}_2$ are realizable then 
${\goth K}_1\rtimes{\goth K}_2$, if it is a (binomial) configuration then is realizable as well}
is false, in general. It suffices to present $\VerSpace(4,3)$ as the ``sum''
of projectively realizable structures $\VerSpace(3,3)$ and $\VerSpace(4,2)$.
So, a natural question arises
\begin{prob}\label{prob:rozmnoz}
  Assume that ${\goth K}_1$ and ${\goth K}_2$ are projective (binomial) configurations
  which satisfy corresponding `recursive equation'
  \begin{equation}
    {\goth K}_1\in\ginconfx(k,m-1) \text{ and } {\goth K}_2\in\ginconfx(k-1,m)
    \text{ for some } k,m\geq 2.
  \end{equation}
  Then there is a bijection 
  $\infty\colon\text{ lines of }{\goth K}_1\longrightarrow \text{ points of }{\goth K}_2$
  so as 
  ${\goth K}_1 \rtimes_\infty{\goth K}_2 \in\ginconfx(k,m)$.
  This observation enables us to construct `Pascal Triangle of Configurations' from, practically,
  arbitrary boundary sequences of configurations, considering arbitrary $\infty$'s.
  \par
  For which maps $\infty$ (is there
  necessarily at least one) the structure ${\goth K}_1\rtimes_\infty{\goth K}_2$
  is projective?\myend
\end{prob}
Note that ``boundary'' sequences $\ginconfx(2,k)$ and $\ginconfx(k,2)$ are known:
$\ginconfx(2,k) = \{ \dual{K_{k+1}} \}$ and  
$\ginconfx(k,2) = \{ K_{k+1} \}$,
and these two sequences consist of projective configurations.

So, considering configurations decomposed with the following schemes
\begin{ctext}
  $\ginconfx(3,k) = \ginconfx(3,k-1)\rtimes \ginconfx(2,k) = 
  \ginconfx(3,k-1)\rtimes \dual{K_{k+1}}$,
  $\ginconfx(k,3) = \ginconfx(k,2)\rtimes \ginconfx(k-1,3) = 
  K_{k+1} \rtimes \ginconfx(k-1,3)$.
\end{ctext}
the real problem lies in the classification/choice of bijections $\infty$!

In particular, there are known binomial partial Steiner triple systems not in the
families $\VerSpace(?,?)$ nor among $\GrasSpace(?,?)$, and nor among $\dual{\VerSpace(?,?)}$
which are projective,
for example, so called quasi-Grassmannians of \cite{skewgras}.
Each such structure $\vergras_n$ has parameters as the corresponding $\GrasSpace(n,2)$.
So, there arises a very particular, but intriguing
\begin{prob}
  Is there a map $\infty$ such that the structure 
  $\vergras_{n-1}\rtimes_\infty\GrasSpace(n-1,3)$ (which has the parameters of 
  $\GrasSpace(n,3)$) is realizable in a Desarguesian projective space.\myend
\end{prob}
\section*{Addendum}
The paper is a result of discussions during Combinatorics 2018 in Arco.


\bigskip
\begin{small}
\noindent
Authors' address:
\\
Krzysztof Pra{\.z}mowski,
\\
Institute of Mathematics, University of Bia{\l}ystok
\\
K. Cio{\l}kowskiego 1M, 15-245 Bia{\l}ystok, Poland
\\
e-mail: \verb+krzypraz@math.uwb.edu.pl+,
\end{small}

\end{document}